\documentclass{article}
\usepackage{amssymb}

\newtheorem{theorem}{Theorem}[section]

\newtheorem{definition}{Definition}[section]

\newtheorem{problem}{Problem}
\newtheorem{proposition}{Proposition}[section]

\newenvironment{proof}[1][Proof]{\noindent\textbf{#1.} }{\ \rule{0.5em}{0.5em}}

\input{tcilatex}

\begin{document}

\title{Geometrical Equivalence and Action Type Geometrical Equivalence of
Group Representations}
\author{J. Sim\~{o}es da Silva, \and A. Tsurkov. \\
Mathematical Department, CCET,\\
Federal University of Rio Grande do Norte (UFRN),\\
Av. Senador Salgado Filho, 3000,\\
Campus Universit\'{a}rio, Lagoa Nova, \\
Natal - RN - Brazil - CEP 59078-970,\\
josenildosimoesdasilva@gmail.com,\\
arkady.tsurkov@gmail.com}
\maketitle

\begin{abstract}
The universal algebraic geometry of group representations was considered in 
\cite{PT}. Thereat the concepts of geometrical equivalence and action type
geometrical equivalence of group representations were defined. It was proved 
\cite[Corollary 2 from Proposition 4.2.]{PT} that if two representations are
geometrically equivalent then they are action type geometrically equivalent.
Also it was remarked \cite[Remark 5.1]{PT} that if two representations $%
(V_{1},G_{1})$ and$\ (V_{2},G_{2})$ are action type geometrically equivalent
and groups $G_{1}$ and $G_{2}$ are geometrically equivalent, the
representations $(V_{1},G_{1})$ and$\ (V_{2},G_{2})$ are not necessarily
geometrically equivalent. But some specific counterexample was not
presented. In this paper we present the example of two representations $%
(V_{1},G_{1})$ and$\ (V_{2},G_{2})$ which are action type geometrically
equivalent and groups $G_{1}$ and $G_{2}$ are geometrically equivalent, but
the representations $(V_{1},G_{1})$ and$\ (V_{2},G_{2})$ are not
geometrically equivalent.
\end{abstract}

\section{Introduction}

\setcounter{equation}{0}

All definitions of the basic notions of the universal algebraic geometry can
be found, for example, in \cite{PlotkinVarCat}, \cite{PlotkinNotions}, \cite%
{PlotkinSame} and \cite{PP}. Also, there are fundamental papers \cite{BMR}, 
\cite{MR} and \cite{DMR2}, \cite{DMR5}.

Some problems of the universal algebraic geometry of many-sorted universal
algebras were considered also in \cite{PT}, \cite{ShestTsur}, \cite%
{TsurkovManySortes} and \cite{TsurkovRepLie}.

We consider the representations of groups over vector spaces over an
arbitrary fixed field $K$. This field is fixed in all our considerations. In
this article the representation of group is a pair $(V,G)$, where $V$ is a
vector space over field $K$ and $G$ is a group. The signature of this
algebraic object includes all the operations in the vector spaces $V$
(multiplication by every scalar $\lambda \in K$ we consider as unary
operation), all the operations in the group $G$ and the operation of action
of the $G$ on vector spaces $V$. We denote this operation by $\circ $:%
\[
V\times G\ni \left( v,g\right) \rightarrow v\circ g\in V. 
\]%
From here and below we will write briefly "representation" instead
"representation of group". The homomorphism $(\alpha ,\beta )$: $%
(V,G)\rightarrow (W,H)$ from representation $(V,G)$ to representation $(W,H)$
is a pair, where $\alpha :V\rightarrow W$ is a linear mapping and $\beta
:G\rightarrow H$ is homomorphism of groups, such that for every $v\in V$ and
every $g\in G$ the equality 
\[
\alpha (v\circ g)=\alpha (v)\circ \beta (g) 
\]%
holds. The reader can see that we consider representations as $2$-sorted
universal algebras: the first sort is a sort of vectors of vector spaces and
the second sort is a sort of elements of groups. This approach to the
representations the reader can find also in \cite{PV}, \cite{Vovsi}, \cite%
{PT}, \cite{TsurkovManySortes}.

The variety of all the representations of groups over a fixed field $K$ will
be denoted by $REP-K$.

As usually, we say that

\begin{definition}
The representation $W\left( X,Y\right) =\left( U\left( X,Y\right) ,H\left(
X,Y\right) \right) $ is called the \textbf{free representation generated by
the sets }$X$\textbf{\ and }$Y$ if, for every $\left( V,G\right) \in REP-K$
and every mappings $f_{1}:X\rightarrow V$ and $f_{2}:Y\rightarrow G$, there
is only one homomorphism of representations $\left( \alpha ,\beta \right)
:\left( U\left( X,Y\right) ,H\left( X,Y\right) \right) \rightarrow \left(
V,G\right) $, such that $\alpha _{\mid X}=f_{1}$, $\beta _{\mid Y}=f_{2}$.
\end{definition}

It is easy to prove \cite{PV}, that $\left( U\left( X,Y\right) ,H\left(
X,Y\right) \right) =(XKF(Y),F(Y))$, where $F(Y)$ is the free group with the
free set of generators $Y$, $KF(Y)$ is the group ring over this group and $%
XKF(Y)=\bigoplus\limits_{x\in X}xKF(Y)$ is the free $KF(Y)$-module with the
free basis $X$.

\section{Basic notions of the algebraic geometry of representations.}

From here and below we suppose that the sets $X$ and $Y$ of generators of
the free representations $(XKF(Y),F(Y))$ are finite. Every equation in the
algebraic geometry of representation is equivalent to the equation $v=0$,
where $v\in XKF(Y)$ or to the equation $f=1$, where $f\in F(Y)$ for
arbitrary $X$ and $Y$. So, in algebraic geometry of representations we can
consider the system of equations $T=(T_{1},T_{2})$, where $T_{1}\subseteq
XKF(Y)$, $T_{2}\subseteq F(Y)$, or, briefly, $(T_{1},T_{2})\subseteq
(XKF(Y),F(Y))$. If we will resolve this system of equations in the
representation $\left( V,G\right) \in REP-K$, then the set $\mathrm{Hom}%
((XKF(Y),F(Y)),(V,G))$ has for us the role of the affine space. The solution
of the system $(T_{1},T_{2})$ in $\left( V,G\right) $ will be the set%
\[
(T_{1},T_{2})_{(V,G)}^{^{\prime }}= 
\]%
\[
\{(\alpha ,\beta )\in \mathrm{Hom}((XKF(Y),F(Y)),(V,G))\ |\ T_{1}\subseteq
\ker (\alpha ),T_{2}\subseteq \ker (\beta )\}\text{.} 
\]%
The algebraic $(V,G)$-closure of the system $(T_{1},T_{2})$ will be the set 
\[
(T_{1},T_{2})_{(V,G)}^{^{\prime \prime }}=(\bigcap\limits_{(\alpha ,\beta
)\in (T_{1},T_{2})_{(V,G)}^{^{\prime }}}\ker (\alpha ),\bigcap_{(\alpha
,\beta )\in (T_{1},T_{2})_{(V,G)}^{^{\prime }}}\ker (\beta ))\subseteq 
\]%
\[
(XKF(Y),F(Y)). 
\]%
This is the maximal system of equations, which has the same solutions as the
system $(T_{1},T_{2})$.

\begin{definition}
Let $(V_{1},G_{1}),(V_{2},G_{2})\in REP-K$. We say that $(V_{1},G_{1})$ and $%
(V_{2},G_{2})$ are \textbf{geometrically equivalent} if $%
(T_{1},T_{2})_{(V_{1},G_{1})}^{^{\prime \prime
}}=(T_{1},T_{2})_{(V_{2,},G_{2})}^{^{\prime \prime }}$, for every $%
(T_{1},T_{2})\subseteq (XKF(Y),F(Y))$ and every $X$ and $Y$. We use the
notation $(V_{1},G_{1})\sim (V_{2},G_{2})$.
\end{definition}

The notion of geometric equivalence can be defined in arbitrary variety of
universal algebras, as one-sorted, as many-sorted. The reader can consult in 
\cite{PlotkinVarCat}.

In our theory

\begin{definition}
The logic formulas which have the form:%
\begin{equation}
(\bigwedge\limits_{i=1}^{n}w_{i})\ \Rightarrow \ w_{0},  \label{qid}
\end{equation}%
where $w_{i}$ can be either $(v_{i}=0)$ or $(f_{i}=1)$, for $v_{i}\in XKF(Y)$
or $f_{i}\in F(Y)$,$\ 0\leq {i\leq {n}}$, $n\in \mathbb{N}$, are called 
\textbf{quasi-identities.}
\end{definition}

\begin{definition}
Let $(V,G)\in REP-K.$ We say that $(V,G)$ \textbf{fulfills} the
quasi-identity (\ref{qid}) if, for every $(\alpha ,\beta )\in \mathrm{Hom}%
((XKF(Y),F(Y)),(V,G))$ such that $\alpha (v_{i})=0$, when $w_{i}$ is $%
(v_{i}=0)$, or $\beta (f_{i})=1$, when $w_{i}$ is $(f_{i}=1)$, $1\leq {i\leq 
{n}}$, we have that $\alpha (v_{0})=0$, when $w_{0}$ is $(v_{0}=0)$, or $%
\beta (f_{0})=1$, when $w_{0}$ is $(f_{0}=1)$. We denote: 
\[
(V,G)\vDash ((\bigwedge\limits_{i=1}^{n}w_{i})\ \Rightarrow \ w_{0}). 
\]
\end{definition}

By \cite[Theorem 2]{PPT} we have the

\begin{proposition}
\label{ge_qid}Let $(V_{1},G_{1}),(V_{2},G_{2})\in REP-K$ and $%
(V_{1},G_{1})\sim (V_{2},G_{2})$ then $(V_{1},G_{1})$ and $(V_{2},G_{2})$
fulfill same\ quasi-identities.
\end{proposition}

By \cite[Proposition 13]{PlotkinVarCat} we also have the

\begin{proposition}
\label{apr_ge}Let $\Theta $ some variety of universal algebras, $%
H_{1},H_{2}\in \Theta $ are finitely generated. Then $H_{1}\sim H_{2}$ if
and only if exist injections $H_{1}\hookrightarrow H_{2}^{I_{2}}$ and $%
H_{2}\hookrightarrow H_{1}^{I_{1}}$, where $I_{1},I_{2}$ are some sets of
indexes and $H_{1}^{I_{1}},H_{2}^{I_{2}}$ corresponding Cartesian powers of
algebras $H_{1}$ and $H_{2}$.
\end{proposition}

In \cite{PT} the action type algebraic geometry of representations was also
considered. This geometry was elaborated on in order to avoid the influence
of the algebraic geometry of the acting group on the algebraic geometry of
representation.

In this geometry we consider only system of action type equations, i.e.,
system of equations which have form $T\subseteq XKF(Y)$. The set of
solutions of this system in the representation $(V,G)$ is the set%
\[
T_{(V,G)}^{^{\prime }}=\{(\alpha ,\beta )\in \mathrm{Hom}%
((XKF(Y),F(Y)),(V,G))\ |\ T\subseteq \ker (\alpha )\}\text{.} 
\]%
The action type $(V,G)$-closure of the system of equations $T$ is a set 
\begin{equation}
T_{(V,G)}^{^{\prime \triangledown }}=\bigcap\limits_{(\alpha ,\beta )\in
T_{(V,G)}^{^{\prime }}}\ker (\alpha )\subseteq XKF(Y).  \label{atc}
\end{equation}%
This is the maximal system of action type equations, which has the same
solutions as the system $T$.

\begin{definition}
Let $(V_{1},G_{1}),(V_{2},G_{2})\in REP-K$. We say that $(V_{1},G_{1})$ and $%
(V_{2},G_{2})$ are \textbf{action type geometrically equivalent} if $%
T_{(V_{1},G_{1})}^{^{\prime \triangledown }}=T_{(V_{2},G_{2})}^{^{\prime
\triangledown }}$, for every $T\subseteq XKF(Y)$ and every $X$ and $Y$. We
use the notation $(V_{1},G_{1})\sim _{at}(V_{2},G_{2})$.
\end{definition}

By \cite[Corollary 2 from Proposition 4.2]{PT}, if two representations $%
(V_{1},G_{1})$ and $(V_{2},G_{2})$ are geometrically equivalent then they
are action type geometrically equivalent.

If $(V,G)\in REP-K$, for every $v\in V$ we can consider the \textit{%
stabilizer} of $v$, defined by

\[
stab(v)=\left\{ g\in G\ |\ v\circ g=v\right\} , 
\]%
and the $\ker (V,G)$, defined by

\[
\ker (V,G)=\bigcap\limits_{v\in V}stab(v). 
\]%
It is easy to check that $\ker (V,G)$ is a normal subgroup of $G$. We denote
by $\tilde{G}$ the quotient group $G/\ker (V,G)$ and by $\sigma $ the
natural epimorphism $\sigma :G\rightarrow G/\ker (V,G)$. It also is easy
check that we obtain the representation $(V,\tilde{G})$ over the vector
spaces $V$ if we define the action of the group $\tilde{G}$ over the vector
spaces $V$ thusly%
\[
v\circ \sigma (g)=v\circ g, 
\]%
where $v\in V$,$\ g\in G$.

\begin{definition}
The representation $(V,\tilde{G})$ is called the \textbf{faithful image} of
the representation $(V,G)$.
\end{definition}

By \cite[Corollary 4 from Theorem 5.1]{PT}, we have the

\begin{proposition}
\label{atefi}Every representation $(V,G)\in REP-K$ is action type
geometrically equivalent to its faithful image $(V,\tilde{G})$.
\end{proposition}

\section{The relation between geometrical equivalence and action type
geometrical equivalence of\protect\linebreak group representations.}

In this Section we will discuss the following question:

\begin{problem}
\label{main}Let $(V_{1},G_{1}),(V_{2},G_{2})\in {REP-K}$. Can we conclude
the geometrical equivalence $(V_{1},G_{1})\sim (V_{2},G_{2})$ from the
action type geometrical equivalence $(V_{1},G_{1})\sim _{at}(V_{2},G_{2})$
and the geometrical equivalence of groups $G_{1}\sim {G_{2}}$?
\end{problem}

The negative answer to this question was remarked by \cite[Remark 5.1]{PT},
but no counterexample was presented. This question has a sense, because in
the action type algebraic geometry of representations we only consider the
specific systems of equations which have form $T\subseteq XKF(Y)$ and the
specific form of the algebraic closure (\ref{atc}). By this restriction we
avoid the influence of the algebraic geometry of action groups. But one
question which comes naturally is whether we are losing some important
information about representations by this restriction. The negative answer
to the Problem \ref{main} shows that we indeed lose some information.

\begin{theorem}
Let $(V_{1},G_{1}),(V_{2},G_{2})\in {REP-K}$. From conditions

\begin{enumerate}
\item $(V_{1},G_{1})\sim _{at}(V_{2},G_{2})$ and

\item $G_{1}\sim {G_{2}}$

we can not conclude $(V_{1},G_{1})\sim (V_{2},G_{2})$.
\end{enumerate}
\end{theorem}

\begin{proof}
We consider a vector space $V$ over arbitrary field $K$, such that $\dim
(V)=2$, with the basis $\{e_{1},e_{2}\}$. We also consider the groups $%
G_{1}=\left\langle a\right\rangle \cong \mathbb{Z}_{2}$ and $%
G_{2}=\left\langle a\right\rangle \times \left\langle b\right\rangle \cong 
\mathbb{Z}_{2}\times \mathbb{Z}_{2}$. We define the action of $G_{1}$ and $%
G_{2}$ over $V$ thusly: 
\[
e_{1}\circ {a}=e_{2},\ e_{2}\circ {a}=e_{1}, 
\]%
\[
e_{1}\circ {b}=e_{1},\ e_{2}\circ {b}=e_{2}. 
\]

Hence, we obtain the representations $(V_{1},G_{1})$ and $(V_{2},G_{2})$. $%
\ker (V,G_{2})=\left\langle b\right\rangle \cong \mathbb{Z}_{2}$. $%
G_{2}/\ker (V,G_{2})\cong G_{1}=\left\langle a\right\rangle \cong \mathbb{Z}%
_{2}$. Therefore, $(V,\widetilde{G_{2}})$ the faithful image of the
representation $(V,G_{2})$ is isomorphic to the representation $%
(V_{1},G_{1}) $. So, by Proposition \ref{atefi} $(V,G_{1})\sim
_{at}(V,G_{2}) $.

The injections $G_{1}\hookrightarrow G_{2}$ and $G_{2}\hookrightarrow
G_{1}\times G_{1}$ exist. Therefore, by Proposition \ref{apr_ge}, $G_{1}\sim 
{G_{2}}$.

\bigskip Now, we consider the quasi-identity%
\[
(x\circ y-x=0)\Rightarrow (y=1). 
\]%
We have that 
\[
(V,G_{1})\vDash ((x\circ y-x=0)\Rightarrow (y=1)), 
\]%
because $\ker (V,G_{1})=\left\{ 1\right\} $, and 
\[
(V,G_{2})\nvDash ((x\circ y-x=0)\Rightarrow (y=1))\text{,} 
\]%
because $\ker (V,G_{2})\neq \left\{ 1\right\} $. It means by Proposition \ref%
{ge_qid} that $(V_{1},G_{1})\nsim (V_{2},G_{2})$.
\end{proof}

\section{Acknowledgements}

We acknowledge the support of Coordena\c{c}\~{a}o de Aperfei\c{c}oamento de
Pessoal de N\'{\i}vel Superior - CAPES (Coordination for the Improvement of
Higher Education Personnel, Brazil). 

We are thankful to Prof. E. Aladova for her important remarks, which helped
a lot in writing this article.

\end{document}